\documentclass[11pt]{amsart}

\usepackage{caption}
\DeclareCaptionLabelFormat{period}{Table \thetable.}
\usepackage[matrix,arrow,curve]{xy}
\usepackage{amsmath}
\usepackage{amssymb}
\usepackage{amsthm}
\usepackage{stmaryrd}
\usepackage[matrix,arrow,curve]{xy}
\usepackage{bbm}
\usepackage{float}
\usepackage{graphicx}
\usepackage{tikz}
\usepackage{tikz-cd}
\usetikzlibrary{calc}
\usetikzlibrary{arrows.meta}

\usepackage[colorlinks=true]{hyperref}

\renewcommand{\L}{\mathcal{L}}

\renewcommand{\O}{\mathcal{O}}

\newcommand{\Z}{\mathbb{Z}}

\renewcommand{\dim}{\operatorname{dim}}
\newcommand{\Hom}{\operatorname{Hom}}
\newcommand{\Ext}{\operatorname{Ext}}
\newcommand{\RHom}{\operatorname{\textbf{R}Hom}}

\newcommand{\Coh}{\operatorname{Coh}}
\newcommand{\syz}{\operatorname{syz}}
\newcommand{\cosyz}{\operatorname{cosyz}}
\newcommand{\st}{\operatorname{st}}
\newcommand{\MCM}{\underline{\operatorname{MCM}}}

\newcommand{\cone}{\operatorname{cone}}
\newcommand{\CR}{\operatorname{CR}}

\newcommand{\depth}{\operatorname{depth}}

\newtheorem{Lem}{Lemma}[section]

\newtheorem{Cor}[Lem]{Corollary}
\newtheorem{Thm}[Lem]{Theorem}
\newtheorem*{Rem}{Remarks}

\def\gr{\operatorname{gr}\!}
\def\grproj{\operatorname{grproj}\!}

\title{Koszul duality between Betti and Cohomology numbers in Calabi-Yau case}
\author{Alexander Pavlov}
\address{University of Wisconsin, 480 Lincoln Dr., Madison, WI, USA}
\email{pavlov@math.wisc.edu}
\date{}

\subjclass[2010]{Primary: 
13D02  
Secondary:
14F05, 
14J32.  
}

\keywords{Betti numbers, Gorenstein ring, Koszul Duality, Maximal Cohen-Macaulay modules}

\begin{document}

\begin{abstract}
Let $X$ be a smooth projective Calabi-Yau variety and $L$ a Koszul line bundle on $X$. We show that for Betti numbers of a maximal Cohen-Macaulay module over the homogeneous coordinate ring $A$ of $X$ there are formulas similar to the formulas for cohomology number. This similarity is realized via the box-product resolution of the diagonal $\Delta_X \subset X \times X$.
\end{abstract}

\maketitle

\section{Introduction}

Let $X$ be a smooth projective variety and $L$ is a very ample line bundle and $A=\bigoplus_{m \geq 0} H^0(X, L^m)$ is the homogeneous coordinate ring. The affine cone over $X$ is an isolated singularity. With this singularity one can associate the triangulated category of singularity, in commutative algebra this category also known under the name of stable category of maximal Cohen-Macaulay modules.

For Calabi-Yau varieties it was shown by Orlov \cite{Orlov09} that the singularity category of $A$ is equivalent as to the bounded derived category $D^b(X)=D^b(\Coh (X))$ of coherent sheaves on $X$.

With a coherent sheaf $F$ on $X$ or more generally an object of $D^b(X)$ one can associate the following numerical invariants
$$
h^{ij}(F)=\dim \mathbb{H}^i(X, F(j)),
$$
we call this the cohomology numbers of $F$. With a finitely generated module $M$ over $A$ one can associate numerical invariants known as Betti numbers
$$
\beta_{ij}(M)=\dim \operatorname{Tor}_A^i(M,k(j)).
$$

The goal of this paper is to show that if the homogeneous coordinate ring $A$ is additionally Koszul then there are 
formulas for Betti numbers related to the definition of cohomology numbers in the same way as Koszul dual algebra $B=A^!$ is related to $A$.

Note that if $L$ is any ample line bundle then a sufficiently large power $L^d$ is a Koszul line bundle meaning that corresponding homogeneous coordinate ring is Koszul. 

A geometric manifestation of Koszulity of the homogeneous coordinate ring $A$ is the resolution of the structure sheaf of the diagonal $\Delta_X \subset X \times X$ that has the box-product structure in each term

\[
\begin{split}
&\dots \to \O(-m) \boxtimes R_m \to \O(-m+1) \boxtimes R_{m-1} \to \dots \\ 
&\to \O(-1) \boxtimes R_1 \to \mathcal{O}_X \boxtimes \mathcal{O}_X
\to \mathcal{O}_{\Delta_X} \to 0,
\end{split}
\]
where $R_m$ are sheaves on $X$ and the box-product is by definition $F \boxtimes G = \pi_1^*(F) \otimes_{X\times X} \pi_2^*(G)$, where $\pi_1$ and $\pi_2$ are projections of $X \times X$ to $X$. 

If we rewrite formula for cohomology numbers as follows
$$
h^{ij} (F)= \dim \Hom (\O(-j),F[i]),
$$
where $\Hom$ stands for $\Hom$-space in the derived category $D^b(X)$. Then (up to a shift) Betti numbers are given by the same formula, but $\O(-j)$ need to be replaced by its partner $R_{-j}$ in box-product resolution above. 

More precisely we have the following 
\begin{Thm}
Let $X$ be a smooth projective Calabi-Yau variety, and let $L$ be a Koszul line bundle s.t. $H^i(X, \O(j))=0$ for $i,j>0$. If $F \in D^b(X)$ and $M$ is the corresponding MCM module then
$$
\beta_{ij}(M)=\dim \Hom (R_{-j}, F[d-1+j-i]),
$$
for $i,j \in \Z$. 
\end{Thm} 

The papers is organized as follows. We give brief preliminaries on maximal Cohen-Macaulay modules and singularity category, details can be found in \cite{Buch87} and \cite{Orlov09}. Then we summarize application of Orlov's equivalence to Betti numbers as it was developed in \cite{P}. For the resolution of the diagonal we closely follow \cite{Kawamata2002}.

\section{Maximal Cohen-Macaulay modules and Singularity Category}

Let $A = \bigoplus_{m \geq 0}  A_m$ be a graded commutative algebra over an algebraically closed field $k$ of characteristic zero. In this paper $A$ is connected that is $A_0=k$ and Gorenstein
$$
\RHom_A(k,A) \cong k(a)[-n],
$$
where the parameter $a \in \Z$ is called the Gorenstein parameter of $A$.

If $X$ is a smooth Calabi-Yau projective variety and $L$ is a very ample line bundle on $X$ then the homogeneous coordinate ring 
$$
A= \bigoplus_{m \geq 0} A_m = \bigoplus_{m \geq 0} H^0(X, L^m)
$$
is generated in degree on, connected, Gorenstein with parameter $a=0$. 

The abelian category of coherent sheaves $\operatorname{Coh}(X)$ is equivalent, by Serre's theorem, to the quotient category $\operatorname{qgr} (R)$
$$
\operatorname{Coh}(X) \cong \operatorname{qgr} (A),
$$
here the abelian category $\operatorname{qgr} (A)$ is defined as a quotient of the abelian category of finitely generated graded $A$-modules by the Serre subcategory of torsion modules
$$
\operatorname{qgr} (A) = \operatorname{mod}_{\gr} (A) / \operatorname{tors}_{\gr} (A),
$$
where $\operatorname{tors}_{\gr} (A)$ is the category of graded modules that are finite dimensional over $k$. 

The bounded derived category of the abelian category of finitely generated graded $A$-modules $D^b(\gr-A) = D^b(\operatorname{mod}_{\gr} (A))$ has a full triangulated subcategory consisting of objects that are isomorphic to bounded complexes of projectives. The latter subcategory can also be described as the derived category of the exact category of graded projective modules, we denote it $D^b(\grproj-R)$. The triangulated {\it singularity category} of $R$ is defined as the Verdier quotient
$$
D^{\gr}_{Sg}(A) = D^b(\gr-A) / D^b(\grproj-A).
$$

In commutative algebra the singularity category is known under the name {\it stable category of maximal Cohen-Macaulay modules}. It was introduced and studied by R.-O. Buchweitz \cite{Buch87}. 

If $M$ is a finitely generated module over $A$, then depth of $M$ is bounded by the dimension of $M$
$$
\depth(M) \leq \dim(M) \leq \dim(A).
$$
If $\depth(M)=\dim(M)$, then the module $M$ is called Cohen-Macaulay, and if $\depth(M)=\dim(A)$ the module $M$ is called a maximal Cohen-Macaulay (MCM). In the stable category of MCM modules objects are MCM modules and stable homomorphism groups are defined by
$$
\underline{\Hom}(M,N)=\Hom(M,N)/\{\text{morphisms that factor through a projective\}}.
$$
In particular, all projective modules (and projective summands) in the stable category are identified with the zero module. We denote this category $\MCM_{\gr}(A)$. Note that the operation of taking syzygies is functorial on the stable category, we denote this functor $\syz$\,. The functor $\syz$ is an autoequivalence of $\MCM_{\gr}(A)$, the inverse of this functor we denote $\cosyz$\,. Moreover, if $A$ is an isolated singularity the stable category of MCM modules is $\Hom$-finite and admits a structure of triangulated category with autoequivalence $\cosyz$.

One of the main results of \cite{Buch87} is the following 
\begin{Thm} 
There is an equivalence of triangulated categories
$$
\MCM_{\gr}(A) \cong D^{\gr}_{Sg}(A)\,.
$$
\end{Thm}

The triangulated singularity category of $A$ is related to the bounded derived category of $X$, as was shown in \cite{Orlov09}. We briefly summarize the result in one special case used in this paper.

The truncated derived global sections functor $\textbf{R} \Gamma_{\geq l}: D^b(X) \to D^b(\gr-A_{\geq l})$ is given by the following direct sum:
$$
\textbf{R} \Gamma_{\geq l}(F)=\bigoplus_{i \geq l} \RHom(\mathcal{O}_X, F(i)),
$$
where $D^b(\gr-A_{\geq l})$ denotes the derived category of the category of finitely generated modules concentrated in degrees $l$ and higher.

The functor $\st : D^b(\gr-A_{\geq i}) \to D^{\gr}_{Sg}(A)$ is the stabilization functor constructed by Orlov.

The special case of Orlov's theorem (see Theorem 2.5 in \cite{Orlov09}) that we need in this paper can be formulated as follows.
\begin{Thm}
Let $X$ be a smooth projective Calabi-Yau variety, $A$ its homogeneous coordinate ring. Then for any parameter $l \in \Z$ the composition of functors
$$
\Phi_l = \st \circ \textbf{R} \Gamma_{\geq l} : D^b(X) \to D^{\gr}_{Sg}(A)
$$
is an equivalence of triangulated categories.
\end{Thm}

Set $\Phi=\Phi_1$.

Abusing notation, we denote the composition of Orlov's and Buchweitz's equivalences by the same $\Phi$.
$$
\Phi : D^b(X) \to \MCM_{\gr}(A).
$$
This equivalence of categories is our main tool for computing Betti numbers of MCM modules.

\section{The tautological formula for Betti numbers}

Let $M$ be a finitely generated graded $A$-module, and let $P^\bullet \to M$ be a minimal free resolution of $M$ over $R$
$$
\dots \rightarrow P^1 \rightarrow P^0 \rightarrow M \rightarrow 0,
$$
where each term $P^i$ of the resolution is a direct sum of free modules with generators in various degrees
$$
P^i \cong \bigoplus_j A(-j)^{\beta_{i,j}}\,.
$$
The exponents $\beta_{i,j}$ are positive integers, that are called the (graded) Betti numbers. The Betti numbers contain information about the shape of a minimal free resolution.

We start with the simple observation that the {\it graded Betti numbers} are given by the formula
$$
\beta_{i,j}(M)=\dim \Ext^i(M, k(-j))\,.
$$
The extension groups $\Ext^i(M, k(-j))$ are isomorphic to the cohomology groups of the complex $\Hom(P^\bullet, k(-j))$, but the resolution $P^\bullet$ is minimal, therefore, differentials in the complex $\Hom(P^\bullet, k(-j))$ vanish.

By the depth lemma $k$-syzygy of a finitely generated module are MCM modules for $k \geq k_0>0$, so up to finitely many terms problem of computing Betti numbers is equivalent to a problem of computing Betti numbers of MCM modules. In this paper we deal with the MCM case. 

In Orlov's equivalence the stable category of MCM modules is used, so our next step is to express the Betti numbers in terms of dimensions of stable extension groups.
\begin{Lem}
Let $M$ be a finitely generated MCM module. Then the graded Betti numbers are equal to the dimensions of the following stable extension groups
$$
\beta_{i,j}(M)=\dim\underline{\Ext}^i(M, k^{st}(-j)),
$$
for $i \geq 0$\,.
\end{Lem}
\begin{proof}
See \cite{P} for the proof.
\end{proof}

\begin{Rem}
The formula above also can be used to compute graded Betti numbers of an MCM module $M$ for $i < 0$. This means that we replace the projective resolution of $M$ by a complete minimal resolution
$$
\CR(M)^\bullet = \ldots \leftarrow P_{-2} \leftarrow P_{-1} \leftarrow P_0 \leftarrow P_1 \leftarrow P_2 \leftarrow \ldots,
$$
and we use such a complete resolution to extend the definition of $\Ext^i(M, k(-j))$ for $i<0$. See \cite{Buch87} for details.
\end{Rem}

Let us denote $\sigma$ the endofunctor $\sigma: D^b(X) \to D^b(X)$ corresponding to the shift in internal degree functor on the MCM side
$$
\sigma=\Phi^{-1} \circ (1) \circ \Phi.
$$

We have the following tautological formula for Betti numbers.

\begin{Thm}\label{main}
The graded Betti numbers of an MCM module $M$ are given by
$$
\beta_{i,j}(M)=\dim \Hom_{D^b(X)}(\Phi^{-1}(M), \sigma^{-j}(\Phi^{-1}(k^{st}))[i])\,.
$$
\end{Thm}

The advantage of the formula is that answers are presented in the derived category, which is easier to understand than stable category of MCM modules.

We can give more detailed description of $\sigma$ and $\Phi^{-1}(k^{\st})$. The following lemma was proved in \cite{P} for elliptic curves but proof for Calabi-Yau varieties goes verbatim.

\begin{Lem}
Let $X$ be a smooth projective Calabi-Yau variety then $\Phi^{-1}(k^{\st}) \cong \O[1]$.
\end{Lem}

Let $E \in D^b(X)$ be an object then we can define the spherical twist functor $T_E$ on $D^b(X)$. It is defined as cone of the derived evaluation map
$$
T_E(F)=\cone(\RHom(E,F) \otimes_k E \to F)
$$

On a Calabi-Yau variety object $E \in D^b(X)$ is called {\it spherical} if $\Ext^\bullet(E, E) \cong H^\bullet(S^{\dim X},k)$. For example, the structure sheaf $\O$ is a spherical object. Spherical objects and functors were introduced in \cite{ST01} and the following important theorem were proved
 \begin{Thm}
 Let $X$ be a smooth projective Calabi-Yau variety. If $E$ is a spherical object, then the spherical twist functor is an auto equivalence.
 \end{Thm}
 
Let $\L :D^b(X) \to D^b(X)$ be an autoequivalence defined by $\L(F)=F \otimes \O(1)$. Finally for functor $\sigma$ we have the following 
 
\begin{Lem}
 There is an isomorphism of functors 
 $$
 \sigma \cong T_\O \circ \L.
 $$
\end{Lem}
\begin{proof}
This is essentially reformulation of Lemma 5.2.1 in \cite{KMVdB11} for $i=1$.
\end{proof}

\section{Resolution of the Diagonal}

In this section we follow Kawamata's presentation \cite{Kawamata2002} of the resolution of the diagonal.

Let $X$ be a projective variety, an ample line bundle $L$ is called {\it Koszul} if the homogeneous coordinate ring
\[
A = \bigoplus_{m \geq 0} H^0(X, L^m)
\]
is Koszul.

We define vector spaces $B_m$ ($m \ge 0$) by $B_0 = A_0$, $B_1 = A_1$ and
\[
B_m = \text{Ker}(B_{m-1} \otimes A_1 \to 
B_{m-2} \otimes A_2)
\]
for $m \ge 2$, where the homomorphism is induced from the multiplication in  
$A$.
We set $A_m = B_m = 0$ for $m < 0$.

One of the ways to say that $A$ is Koszul algebra is to require that 
the sequence of natural homomorphisms 
\begin{equation}\label{seq-module}
\begin{split}
&\dots \to B_m \otimes A(-m) \to 
B_{m-1} \otimes A(-m+1) \to 
\dots \\
&\to B_1 \otimes A(-1) \to A \to k \to 0
\end{split}
\end{equation}
is exact, where the shifted module $A(j)$ is defined by 
$A(j)_k = A_{j+k}$.

By \cite{Backelin}~Theorem~2, the subring 
$A^{(d)} = \bigoplus_{m \geq 0} A_{dm}$ of $A$ is a Koszul algebra
for a sufficiently large integer $d$, i.e. $L^d$ is a Koszul line bundle. 

We define sheaves $R_m$ ($m \ge 0$) on $X$ by $R_0 = \mathcal{O}_X$ and 
\[
R_m = \text{Ker}(B_m \otimes \mathcal{O}_X \to 
B_{m-1} \otimes L)
\]
for $m \ge 1$, where the homomorphism is induced from the natural 
homomorphisms $B_m \to B_{m-1} \otimes A_1$ and 
$A_1 \otimes \mathcal{O}_X \to L$.

If we take the sequence of associated sheaves to
the tensor product of $(\ref{seq-module})$ with $A(m)$,
we obtain an exact sequence
\begin{equation}\label{seq-sheaf}
0 \to R_m \to B_m \otimes \mathcal{O}_X
\to B_{m-1} \otimes L \to 
\dots \to B_1 \otimes L^{m-1} \to L^m \to 0.
\end{equation}

\begin{Lem}\label{AR}
There is an exact sequence
\[
\begin{split}
&0 \to A_0 \otimes R_m \to A_1 \otimes R_{m-1} 
\to \\
&\dots \to 
A_{m-1} \otimes R_1 \to A_m \otimes
R_0 \to L^m \to 0.
\end{split}
\]
\end{Lem}

On the product $X \times X$ 
by Lemma~\ref{AR}, we obtain a homomorphism 
$L^{-m} \boxtimes R_m \to L^{-m+1} \boxtimes R_{m-1}$
as a composition of the following homomorphisms
\begin{equation}\label{LRhom}
L^{-m} \boxtimes R_m \to A_1 \otimes (L^{-m} \boxtimes R_{m-1})
\to L^{-m+1} \boxtimes R_{m-1}.
\end{equation}

\begin{Thm}\label{diagonal1}
Let $X$ be a projective variety, $L$ be a Koszul line bundle, and
$\Delta_X \subset X \times X$ the diagonal subvariety of the direct product.
Then the complex of sheaves on $X \times X$ 
\[
\begin{split}
&\dots \to L^{-m} \boxtimes R_m \to L^{-m+1} \boxtimes R_{m-1} \to \dots \\ 
&\to L^{-1} \boxtimes R_1 \to \mathcal{O}_X \boxtimes \mathcal{O}_X
\to \mathcal{O}_{\Delta_X} \to 0
\end{split}
\]
is exact.
\end{Thm}

Proofs of the lemma and the theorem can be found in \cite{Kawamata2002}.

\section{The Duality formula}

We assume that $X$ is a smooth projective variety and $L$ is an ample line 
bundle on $X$, then for sufficiently large integer $d$ line bundle $\O(1)=L^d$
and all its powers are Koszul and have trivial high cohomology $H^i(X, \O(j))=0$.

The following short exact sequence immediately follows from the definition of the sheaves $R_m$ 
$$
0 \to R_m \to B_m \otimes \O \to R_{m-1}(1) \to 0.
$$

\begin{Lem}
For $m \geq 0$ there is an isomorphism $B_m \cong H^0(X, R_{m-1}(1)).$
\end{Lem}
\begin{proof}
\[
\begin{split}
H^0(X,R_{m-1}(1))=&\ker(B_{m-1} \otimes H^0(X, \O(1)) \to B_{m-2} \otimes H^0(X, \O(2))=\\
&\ker(B_{m-1} \otimes A_1 \to B_{m-2} \otimes A_2)=B_m.
\end{split}
\]
\end{proof}

\begin{Lem}
For $m \geq 0$, $i>0$ and $j > 0$ the cohomology groups $H^i(X, R_m (j))$ vanish
$$
H^i(X, R_m (j)) =0.
$$
\end{Lem}
\begin{proof}
We prove vanishing by induction on $m$. If $m=0$ then $R_m = \O$ 
and vanishing follows from assumptions on $\O(1)$.
If $m>0$ then tensoring short exact sequence 
$$
0 \to R_m \to B_m \otimes \O \to R_{m-1}(1) \to 0.
$$
with $\O(j)$ and considering long exact sequence of cohomology groups
\[
\begin{split}
\ldots \to &H^{i-1}(X, R_{m-1}(j+1)) \to \\
\to H^i(X,R_m(j)) \to H^i(X, \O(j)) \otimes B_m \to &H^i(X, R_{m-1}(j+1)) \to \ldots
\end{split}
\]
we get vanishing for $i>1$ by assumption of induction and assumptions of $\O(1)$. 
For $i=1$ vanishing follows from assumption that $A$ is Koszul.
\end{proof}

Combining these two lemmas we get

\begin{Thm}
Sheaves $R_m$ satisfy recursive relation in the derived category $D^b(X)$ 
$$
R_m[1]=(T_\O \circ \L) (R_{m-1}).
$$
\end{Thm}
\begin{proof}
By definition of spherical twist $T_\O$ we have distinguished triangle
$$
\oplus_i H^i(X, R_{m-1}(1)) \otimes \O \to R_{m-1}(1) \to (T_\O \circ \L) (R_{m-1}) \to \ldots
$$
By two previous lemmas $\oplus_i H^i(X, R_{m-1}(1)) \cong B_m$. Short exact sequence 

$$
0 \to R_m \to B_m \otimes \O \to R_{m-1}(1) \to 0.
$$
gives distinguished triangle
$$
R_m \to B_m \otimes \O \to R_{m-1}(1) \to R_m[1]
$$
Noticing that these two distinguished triangles have the same objects and maps 
we conclude that $(T_\O \circ \L) (R_{m-1}) \cong R_m[1]$.
\end{proof}

We can reformulate the result as follows

\begin{Cor}
Sheaves $R_m$ form the orbit of the functor $\gamma=T_\O \circ \L[-1]$ for $m \geq 0$
$$
R_m = \gamma^m(\O).
$$
\end{Cor}

\begin{proof}
Iterating identity $R_m[1]=(T_\O \circ \L) (R_{m-1})$ we get $R_m[m]=(T_\O \circ \L)^m(\O)$.
\end{proof}

The key observation is that two functors $\sigma$ used in the formula for Betti numbers and $\gamma$ are related by $\sigma=\gamma[1]$.

Let us set by definition $R_m = \gamma^m(\O)$ for $m \in \Z$. For $m<0$ we get objects of the derived category $R_m \in D^b(X)$, for example $R_{-1}=\O(-1)[d]$, where $d=\dim X$.

If $F$ is any object in the derived category $D^b(X)$ we define cohomology numbers of $F$ as the following dimensions of the hypercohomology groups 
$$
h^{ij}(F)=\dim \mathbb{H}^i(X, F(j)),
$$
where $i,j \in \Z$.

Combining all our observations together we get the following version of Koszul duality between cohomology numbers of $F$ and Betti numbers of the MCM module $M$ corresponding to $F$ under Orlov's equivalence. This duality is controlled by two families of objects $\O(i)$ and $R(j)$ appearing in the box-product resolution of the diagonal

\[
\begin{split}
&\dots \to \O(-m) \boxtimes R_m \to \O(-m+1) \boxtimes R_{m-1} \to \dots \\ 
&\to \O(-1) \boxtimes R_1 \to \mathcal{O}_X \boxtimes \mathcal{O}_X
\to \mathcal{O}_{\Delta_X} \to 0.
\end{split}
\]

\begin{Thm}
Let $X$ be a smooth projective Calabi-Yau variety, and let $L$ be a Koszul line bundle s.t. $H^i(X, \O(j))=0$ for $i,j>0$. If $F \in D^b(X)$ and $M$ is the corresponding MCM module then
\[
\begin{split}
&\beta_{ij}(M)=\dim \Hom (R_{-j}, F[d-1+j-i]),\\
&h^{ij} (F)= \dim \Hom (\O(-j),F[i]).
\end{split}
\]
for $i,j \in \Z$. 
\end{Thm} 
\begin{proof}
We know that $\beta_{ij}(M)= \dim \Hom (F, \sigma^{-j}(\O)[i+1])=\dim \Hom(F, R_{-j}[i-j+1])$. Applying Serre's duality we get 
$\beta_{ij}(M)=\dim \Hom (R_{-j}, F[d-1+j-i])$.

The second formula is clear.
\end{proof}


\begin{thebibliography}{99}

\bibitem{Backelin} J. Backelin. \emph{On the rate of growth of the homologies of Veronese subrings}.
Lecture Notes in Math. {\bf 1183}(1986), 79--100.

\bibitem{Buch87} R.-O. Buchweitz, \emph{Maximal Cohen-Macaulay Modules and Tate Cohomology over Gorenstein Rings}, preprint, \url{https://tspace.library.utoronto.ca/handle/1807/16682}.

\bibitem{Kawamata2002}  Y. Kawamata, \emph{Equivalences of derived catgories of sheaves on smooth stacks}, American Journal of Mathematics, Vol. 126, No. 5 (Oct., 2004), pp. 1057--1083

\bibitem{KMVdB11} B. Keller, D. Murfet, M. Van den Bergh \emph{On Two Examples by Iyama and Yoshino}, Compos. Math. 147 (2011), no. 2, 591-612.

\bibitem{Orlov09}  D. Orlov, \emph{Derived Categories of Coherent Sheaves and Triangulated Categories of Singularities}, in "Algebra, Arithmetic and Geometry: In Honor of Yu. I. Manin", Vol 2, Birkh\"{o}user (2009), 503-531.

\bibitem{P} A. Pavlov, \emph{Betti Tables of Maximal Cohen-Macaulay modules over the cone a plane cubic}, \url{http://arxiv.org/abs/1511.05089}

\bibitem{ST01} P. Seidel, R. Thomas, \emph{Braid Group Actions on Derived Categories of Coherent Sheaves}, Duke Math. J. 108(2001), 37-108.


\end{thebibliography}
\end{document}